\documentclass[a4paper,11pt]{article}
\usepackage{graphicx}
\usepackage{amssymb}
\usepackage{amsthm}
\usepackage{amsfonts}
\usepackage{amsmath}
\usepackage{hyperref}
\usepackage{xcolor}

\title{Negative Avoiding Sequences}
\author{Chris J Mitchell and Peter R Wild}
\date{9th November 2026}
\theoremstyle{plain}
\newtheorem{lemma}{Lemma}[section]
\newtheorem{theorem}[lemma]{Theorem}

\theoremstyle{definition}
\newtheorem{definition}{Definition}[section]
\newtheorem{construction}{Construction}[section]
\newtheorem{example}{Example}[section]

\theoremstyle{remark}
\newtheorem{remark}{Remark}[section]

\begin{document}

\maketitle

\begin{abstract}
Negative avoiding sequences of span $n$ are periodic sequences of elements from $\mathbb{Z}_k$ for
some $k$ with the property that no $n$-tuple occurs more than once in a period and if an $n$-tuple
does occur then its negative does not.  They are a special type of cut-down de Bruijn sequence with
potential position-location applications.  We establish a simple upper bound on the period of such
a sequence, and refer to sequences meeting this bound as \emph{maximal} negative avoiding
sequences. We then go on to demonstrate the existence of maximal negative avoiding sequences for
every $k\geq3$ and every $n\geq2$.
\end{abstract}

\section{Introduction} \label{section:introduction}

The main focus of this paper is a special class of infinite periodic sequences with elements from
$\mathbb{Z}_k$, for some $k>2$ --- which we refer to as $k$-ary sequences. We are concerned with
$n$-tuples (or \emph{factors}), i.e.\ ordered strings of $n$ symbols, that are equal to $n$
consecutive symbols from a sequence.  We refer to an $n$-tuple \emph{occurring} in a sequence if
the $n$-tuple equals $n$ consecutive symbols occurring somewhere in the sequence; analogously we
refer to a sequence \emph{containing} an $n$-tuple if that $n$-tuple occurs somewhere in the
sequence.  For a sequence $S = (s_i)$ and $n\geq1$ we write $\mathbf{s}_n(i) =
(s_i,s_{i+1},\ldots,s_{i+n-1})$, for the $n$-tuple corresponding to a substring or factor of $n$
consecutive symbols occurring in the sequence at positions $i,i+1,\ldots,i+n-1$.

If a sequence $S=(s_i)$ has period $m$ then we represent the sequence as $[s_0,s_1,\dots,s_{m-1}]$,
i.e.\ a single period of the sequence contained in square brackets.  In much of the literature
$[s_0,s_1,\dots,s_{m-1}]$ is described as a \emph{cycle}.

We refer to the sequences studied here as \emph{Negative Avoiding Sequences (NASs)}.  Informally, a
Negative Avoiding Sequence of span $n$ is a periodic sequence of elements from $\mathbb{Z}_k$ for
some $k>2$, with two special properties:
\begin{itemize}
\item every possible $n$-tuple of length $n$ occurs at most once in a period of the sequence;
    and
\item if an $n$-tuple $\mathbf{a}=(a_0,a_1,\dots,a_{n-1})$ occurs in the sequence, then the
    $n$-tuple $(-a_0,-a_1,\dots,-a_{n-1})$ does not occur in the sequence.
\end{itemize}

In this paper we provide a simple bound on the period of a NAS, and give constructions establishing
that the bound is tight for all $n\geq2$ and $k\geq3$.

More formally, we can make the following definitions.

\begin{definition}[\cite{Alhakim24a}]
A $k$-ary \emph{$n$-window sequence $S = (s_i)$} is a periodic sequence of elements from
$\mathbb{Z}_k$ ($k>1$, $n>1$) with the property that no $n$-tuple occurs more than once in a
period of the sequence, i.e.\ with the property that if $\mathbf{s}_n(i) = \mathbf{s}_n(j)$ for
some $i,j$, then $i \equiv j \pmod m$ where $m$ is the period of the sequence.
\end{definition}

\begin{remark}
An $n$-window sequence is also known as a cut-down de Bruijn Sequence --- see, for example, Jackson
et al.\ \cite[Problem 477]{Jackson09} and Cameron et al.\ \cite{Cameron25}.
\end{remark}

Note that the well-studied \emph{de Bruijn sequences} of span $n$ (see, for example, \cite[Section
7.2.1.1]{Knuth11}) are simply $n$-window sequences in which every possible $n$-tuple occurs exactly
once in a period.  Such sequences clearly have period $k^n$.

We now make the following key definition.

\begin{definition}
Suppose $k\geq3$ and $n\geq1$. If $\mathbf{a}=(a_0,a_1,\dots,a_{n-1})$ is a $k$-ary $n$-tuple, then
we define its negative in the obvious way, i.e.\
\[ -\mathbf{a}=(-a_0,-a_1,\dots,-a_{n-1}). \]
\end{definition}

The special class of $n$-window sequences of interest here can now be defined.

\begin{definition}
Suppose $k\geq3$ and $n\geq 2$.  A $k$-ary \emph{Negative Avoiding Sequence (NAS)} of span $n$ is
an infinite periodic $n$-window sequence $S=(s_i)$, $i\geq0$, with elements from $\mathbb{Z}_k$
with the property that if the $n$-tuple $\mathbf{a}$ occurs in $S$, then $-\mathbf{a}$ does not
occur in $S$. We refer to such a sequence a $\mathcal{NAS}_k(n)$.
\end{definition}

If a $k$-ary $n$-tuple $\mathbf{a}$ satisfies $\mathbf{a}=-\mathbf{a}$ then we call it
\emph{self-negative}.  Clearly a $\mathcal{NAS}_k(n)$ cannot contain any self-negative $n$-tuples.

\begin{example}
A simple example for $k=3$ and $n=2$ would be the cycle $[0112]$, in which the 2-tuples $01$, $11$,
$12$, $20$ appear, but their negatives, i.e.\ $02$, $22$, $21$ and $10$ do not.  Since $00$ cannot
appear being self-negative, this sequence clearly has maximum possible period.
\end{example}

The remainder of the paper is structured as follows. Motivation for the work and background
definitions are provided in Sections~\ref{section:motivation} and \ref{section:deBruijn}. This is
followed in Section~\ref{section:simple_bound} by an upper bound on the period of negative avoiding
sequences.  Results establishing the existence of maximal period negative avoiding sequences are
provided in Sections \ref{section:maximal_oddn}, \ref{section:maximal_n2} and
\ref{section:maximal_neven}. The paper concludes in Section~\ref{section:conclusions} with a
summary and a discussion of open research questions.

\section{Motivation and related work}  \label{section:motivation}

De Bruijn sequences have a range of real-world applications, ranging from position location on a
linear track or rotating shaft \cite{Burns92,Burns93,Peng23,Petriu98,Yamaguchi05}, to genome
assembly \cite{Compeau11}. As observed by Cameron et al.\ \cite{Cameron25}, `for some applications
it may be more convenient to produce a cycle of arbitrary length such that there are no repeated
strings', i.e.\ an $n$-window sequence of a particular period that derives from the specifics of
the application.

A number of authors have considered the problem of constructing $n$-window sequences with
particular properties. Jackson et al.\ \cite[Problem 477]{Jackson09} observed that a binary
$n$-window sequence exists for every possible period $m$ satisfying $n\leq m\leq2^n$. Sawada et
al. \cite{Sawada11} and Ruskey et al.\ \cite{Ruskey12} showed that a binary $n$-window sequence
can be constructed containing all $n$-tuples with weights in the range $[a,b]$ as long as
$0\leq a<b\leq n$, and such sequences were further studied by Li et al.\ \cite{Li23}.  Later,
Gabri\'{c} et al.~\cite{Gabric20}, generalised this result to alphabets of an arbitrary size in
the case where weight is defined as the digit sum, and showed that all $n$-tuples bounded above
(or below) by a given weight can be arranged in a cut-down de Bruin sequence.  Even more
recently, Gabri\'{c} et al.~\cite{Gabric26} showed how to construct sequences of this type that
can be readily decoded.

Sequences containing only $n$-tuples with specified weights are a particular example of what
are sometimes known as `de Bruijn sequences with forbidden subsequences', as for example
examined in the binary case by Penne~\cite{Penne10} and for arbitrary alphabets by
Alhakim~\cite{Alhakim22} and Sawada et al.~\cite{Sawada16}. Tan and Shallot~\cite{Tan13},
studied the more general problem of the existence of sequences containing only the members of a
pre-specified set of $n$-tuples, where they allow such $n$-tuples to occur multiple times in a
period. These are all examples of \emph{Universal cycles} \cite{Chung92}, i.e.\ periodic
sequences in which all elements of a particular set of $n$-tuples occur precisely once in a
period.

The sequences we consider in this paper are closely related to, yet distinct from, those
previously examined. That is, we consider the case where the $n$-tuples are divided into pairs
(an $n$-tuple and its negative) and we only permit at most one of each pair to occur in a
sequence. Such a sequence might have application in a position-location case where it is
difficult for a reader to distinguish between an $n$-tuple and its negative.  These sequences
are closely related to \emph{orientable} \cite{Dai93,Mitchell26} and \emph{negative orientable}
sequences \cite{Mitchell25a}, which are also periodic $n$-window sequences in which the
$n$-tuples are divided into pairs and at most one of each pair is permitted to occur in a
sequence; the only difference is how the `pairing' function is defined.

\section{The de Bruijn graph and the Lempel homomorphism}  \label{section:deBruijn}

Before proceeding we need to introduce certain key concepts which we use in the remainder of the
paper.

\subsection{The de Bruijn graph}

Following Alhakim et al.\ \cite{Alhakim24a}, for positive integers $n$ and $k$ greater than one,
let $\mathbb{Z}_k^n$ be the set of all $k^n$ tuples of length $n$ with entries from the group
$\mathbb{Z}_k$ of residues modulo $k$.

\begin{definition}
The order $n$ de Bruijn digraph, $B_k(n)$, is a directed graph with $\mathbb{Z}^n_k$ as its vertex
set in which, for any two vertices $\textbf{x} = (x_0,x_1,\ldots,x_{n-1})$ and $\textbf{y} =
(y_0,y_1,\ldots,y_{n-1})$, the pair $(\textbf{x},\textbf{y})$ is an edge if and only if $y_i =
x_{i+1}$ for every $i$ ($0\leq i< n-1$). We label such an edge with the $(n+1)$-tuple
$(x_0,x_1,\ldots,x_{n-1},y_{n-1})$.
\end{definition}

Note that we have defined two ways of specifying an edge in $B_k(n)$, namely as either a pair of
vertices $(\mathbf{x},\mathbf{y})$, where $\mathbf{x},\mathbf{y}$ are $k$-ary $n$-tuples, or as a
single $k$-ary $(n+1)$-tuple $\mathbf{z}$.

\begin{definition}  \label{definition;Eulerian} An \emph{Eulerian digraph} is a
connected digraph for which every vertex has in-degree equal to out-degree --- this latter property
characterises a \emph{balanced} digraph.
\end{definition}

The name derives from the fact that there exists an Eulerian circuit, i.e.\ a path visiting every
edge once, in a digraph if and only if the digraph is Eulerian --- see, for example, Corollary 6.1
of Gibbons~\cite{Gibbons85} or Knuth~\cite[Section 2.3.4.2, Theorem G]{Knuth97}. Moreover, there
are simple and efficient algorithms for finding Eulerian circuits --- see for example~\cite[Figure
6.5]{Gibbons85} or \cite[Section 2.3.4.2, Theorem D]{Knuth97}.

\subsection{Antinegative subgraphs}

It is straightforward to verify that there is a correspondence between a $k$-ary de Bruijn
sequence of order $n$ and a directed Eulerian circuit in $B_{k}({n-1})$, in which consecutive
edges in the circuit correspond to consecutive $n$-tuples occurring in the sequence. Inspired
by this correspondence, and following \cite{Mitchell26}, we next describe a correspondence
between negative avoiding sequences of span $n$ and Eulerian circuits in subgraphs of
$B_{k}(n-1)$, where these subgraphs satisfy certain special properties.  If we can then
construct subgraphs with these special properties which admit Eulerian circuits then we have a
simple method of generating negative avoiding sequences. Note that the graphs we are concerned
with here are examples of Rauzy graphs \cite{Rauzy82}.

\begin{definition}  \label{definition:antinegative}
Suppose $T$ is a subgraph of the de Bruijn digraph $B_{k}(n)$ for some $n\geq2$ and $k\geq2$.  $T$
is said to be \emph{antinegative} if, given vertices
$\mathbf{x}=(x_0,x_1,\ldots,x_{n-1}),\mathbf{y}=(y_0,y_1,\ldots,y_{n-1})\in \mathbb{Z}_k^n$ where
$(\mathbf{x},\mathbf{y})$ is an edge in $T$, then $(-\mathbf{x},-\mathbf{y})$ is \emph{not} an edge
in $T$.
\end{definition}

\begin{definition}[Definition 3.2 of \cite{Mitchell26}]  \label{definition:En(s)}
Suppose $S$ is a periodic $k$-ary $n$-window sequence of period $m$ for some $n\geq2$, $k\geq2$ and
$m\geq1$. Then the \emph{edge-graph} $E_n(S)$ of $S$ is defined to be the subgraph of $B_{k}(n-1)$
whose directed edges correspond to $n$-tuples occurring in a period of $S$, i.e.\ with edge set
\[ \{ \mathbf{s}_n(i)~:~0\leq i\leq m-1  \} \]
and whose vertices are those vertices of $B_{k}(n-1)$ that have in-degree at least one.
\end{definition}

We can now state a key lemma.

\begin{lemma}  \label{lemma:NAS_iff_antinegative}
Suppose $S$ is a $k$-ary periodic $n$-window sequence.  Then $S$ is an $\mathcal{NAS}_k(n)$ if and
only if $E_n(S)$ is antinegative in $B_{k}(n-1)$.
\end{lemma}

\begin{proof}
Suppose $S$ is an $\mathcal{NAS}_k(n)$.  Then, by definition, $\mathbf{s}_n(i) \neq
-\mathbf{s}_n(j)$, for any $i,j$.  Hence, again by definition, this means that $\mathbf{a}\neq
-\mathbf{b}$ for any edges $\mathbf{a},\mathbf{b}$ in $E_n(S)$.  Hence $E_n(S)$ is antinegative.

Now suppose $E_n(S)$ is antinegative, and hence $\mathbf{a}\neq -\mathbf{b}$ for any edges
$\mathbf{a},\mathbf{b}$ in $E_n(S)$.  Thus, $\mathbf{s}_n(i) \neq -\mathbf{s}_n(j)$, for any $i,j$,
and so $S$ is an $\mathcal{NAS}_k(n)$.
\end{proof}

This enables us to give a result which is key to establishing the results in the remainder of
the paper.

\begin{theorem}  \label{theorem:correspondence}
If $S$ is an $\mathcal{NAS}_k(n)$ of period $m$ then $E_n(S)$ is an antinegative Eulerian subgraph
of $B_{k}(n-1)$ containing $m$ edges.  Moreover, if $T$ is an antinegative Eulerian subgraph of
$B_{k}(n-1)$ with $m$ edges, then there exists an $\mathcal{NAS}_k(n)$ $S$ of period $m$ such that
$E_n(S)=T$.
\end{theorem}

\begin{proof}
Suppose $S=(s_i)$ is an $\mathcal{NAS}_k(n)$ of period $m$.  Then, by
Lemma~\ref{lemma:NAS_iff_antinegative}, $E_n(S)$ is antinegative in $B_{k}(n-1)$.  Also, trivially,
$E_n(S)$ contains $m$ edges.  Finally, if $(x_1,x_2,\ldots,x_{n-1})$ is a vertex in $E_n(S)$, then
the incoming edges all have distinct labels of the form $(x,x_1,x_2,\ldots,x_{n-1})$ for some $x$,
and correspond to $\mathbf{s}_n(i)$ for some $i$.  Then the edges corresponding to
$\mathbf{s}_n(i+1)$ will have distinct labels of the form $(x_1,x_2,\ldots,x_{n-1},y)$ for some
$y$, and for every incoming edge there is an outgoing edge, and vice versa.  Hence the in-degree of
every vertex is the same as the out-degree, i.e.\ $E_n(S)$ is balanced.  Finally, note that
$E_n(S)$ is connected since we only include vertices in $E_n(S)$ with in-degree greater than zero,
and the sequence $S$ defines a path incorporating every vertex in $E_n(S)$.

Now suppose $T$ is an antinegative Eulerian subgraph of $B_{k}(n-1)$ with $m$ edges.  Then there
exists an Eulerian circuit of length $m$ in $T$ corresponding to a sequence $S$ with period $m$ in
the natural way.  $S$ is clearly an $n$-window sequence since the circuit visits every edge exactly
once, and each edge corresponds to a unique $n$-tuple.  It is straightforward to verify that
$E_n(S)=T$. Finally, $S$ is an $\mathcal{NAS}_k(n)$ from Lemma~\ref{lemma:NAS_iff_antinegative}.
\end{proof}

\begin{remark}
In general, for any antinegative Eulerian subgraph of $B_{k}(n-1)$ $T$, there may exist many
negative avoiding sequences with edge-graph $T$, since there may be many different Eulerian
circuits in $T$, each corresponding to a different sequence.
\end{remark}

\subsection{Implications}

Theorem~\ref{theorem:correspondence} means that if we can construct an antinegative Eulerian
subgraph of $B_{k}(n-1)$ with $m$ edges for `large' $m$ (i.e.\ with $m$ close to the maximum
possible), then we will immediately have a set of negative avoiding sequences with period close to
the maximum.  As a result, in the remainder of this paper we consider ways of constructing sets of
edges in $B_{k}(n-1)$ that define an antinegative Eulerian subgraph of $B_{k}(n-1)$.

\subsection{The Lempel Homomorphism}

We next define the Lempel homomorphism $D$. Following \cite{Mitchell26} we give a simplified
version of the definition of Alhakim and Akinwande \cite{Alhakim11}, who generalised the original
Lempel definition \cite{Lempel70}.

\begin{definition} \label{Lempel}
Let $D:\mathbb{Z}_k^n\rightarrow \mathbb{Z}_k^{n-1}$ be defined such that
\[ D(\mathbf{a}) = (a_1-a_0,a_2-a_1,\ldots, a_{n-1}-a_{n-2}) \]
where $\mathbf{a} = (a_0,a_1,\ldots,a_{n-1})\in \mathbb{Z}_k^n$. Clearly $D$ is onto; thus we also
define $D^{-1}$ to map an element of $\mathbb{Z}_k^{n-1}$ to the set of its pre-images under $D$.
\end{definition}

We have the following.

\begin{lemma}[Alhakim \& Akinwande, \cite{Alhakim11}] \label{lempel inverse size}
$|D^{-1}(\mathbf{a})|=k$ for every $\mathbf{a}\in \mathbb{Z}_k^{n-1}$.
\end{lemma}

Since the edges of $B_k(n-1)$ are associated with $k$-ary $n$-tuples, $D$ induces a mapping, which
we also write as $D$, from $B_k(n)$ to $B_{k}(n-1)$. It is simple to show that $D$ is a graph
homomorphism. In this context we also write $D^{-1}$ for the pre-image mapping which maps an edge
of $B_{k}(n-1)$ to a set of $k$ edges of $B_{k}(n)$ and each pre-image edge contains corresponding
pre-image vertices.

The following simple result shows the value of $D$ in constructing negative avoiding sequences.

\begin{lemma}  \label{lemma:Lempel_antisymmetry}
Suppose $n\geq2$ and $k\geq3$.  If $T$ is an antinegative subgraph of the de Bruijn digraph
$B_{k}(n-1)$ with edge set $E$, then $D^{-1}(E)$, of cardinality $k|E|$, is the set of edges for an
antinegative subgraph $D^{-1}(T)$ of $B_{k}(n)$. Moreover, the in-degree (out-degree) of any vertex
$\mathbf{x}$ in $D^{-1}(T)$ equals the in-degree (out-degree) of $D(\mathbf{x})$ in $T$; thus, in
particular, if $T$ is balanced then so is $D^{-1}(T)$.
\end{lemma}

\begin{proof}
Suppose $D^{-1}(T)$ is not antinegative, i.e.\ there exist $(n+1)$-tuples
$\mathbf{a}=(a_0,a_1,\ldots,a_{n}),\mathbf{b}=(b_0,b_1,\ldots,b_{n})\in D^{-1}(E)$ such that
$\mathbf{a}=-\mathbf{b}$, i.e.\ $a_i=-b_i$, for $0\leq i\leq n$.

Suppose also that $D(\mathbf{a})=\mathbf{c}$ and $D(\mathbf{b})=\mathbf{d}$, where
$\mathbf{c}=(c_0,c_1,\ldots,c_{n-1}),\mathbf{d}=(d_0,d_1,\ldots,d_{n-1})\in E$. Hence
$c_i=a_{i+1}-a_i$ and $d_i=b_{i+1}-b_i$ for $0\leq i\leq n-1$.  Since $a_i=-b_i$ for $0\leq i\leq
n$, for any $j$ ($0\leq j\leq n-1$) we have:
\[ c_j = a_{j+1}-a_j = -b_{j+1}+b_{j} = -d_j. \]
Hence $\mathbf{c}=-\mathbf{d}$, but this contradicts the assumption that $T$ is antinegative.

Every edge in $E$ corresponds to $k$ edges in $D^{-1}(E)$, and hence $|D^{-1}(E)|=k|E|$.

It remains to show that every vertex of $D^{-1}(T)$ has in-degree equal to its out-degree. Suppose
$\mathbf{x}=(x_1,x_2,\ldots,x_{n})$ is a vertex of $D^{-1}(T)$.  For every edge
$\mathbf{a}=(a,x_1,x_2,\ldots,x_{n})\in D^{-1}(T)$ that ends in $\mathbf{x}$, there is a
corresponding edge $D(\mathbf{a})=(x_1-a,x_2-x_1,\ldots,x_{n}-x_{n-1})\in T$ that ends in
$D(\mathbf{x})$.  That is, the in-degree of $\mathbf{x}$ in $D^{-1}(T)$ will equal the in-degree of
$D(\mathbf{x})$ in $T$.  An exactly similar result holds for out-degree.  Since every vertex of $T$
has in-degree equal to its out-degree, the same holds for $D^{-1}(T)$.
\end{proof}

\section{A simple bound and some examples}    \label{section:simple_bound}

Having established our key underlying ideas, we give the following elementary bound.

\begin{theorem}  \label{theorem:NAS_bound}
Suppose $k\geq3$ and $n\geq 2$.  If $S$ is a $\mathcal{NAS}_k(n)$ of period $m$ then
\[ m\leq\begin{cases}
\frac{k^n-1}{2}& \text{if $k$ is odd}\\
\frac{k^n-2^n}{2}& \text{if $k$ is even}
\end{cases}\]
\end{theorem}

\begin{proof}
Clearly, if an $n$-tuple $\mathbf{a}$  is self-negative, i.e.\ it
satisfies $\mathbf{a}=-\mathbf{a}$, then it cannot occur in a $\mathcal{NAS}_k(n)$.  There is one
such $n$-tuple if $k$ is odd, namely $(0,0,\dots,0)$, and $2^n$ if $k$ is even, namely all
$n$-tuples with entries equal to $0$ or $k/2$.
\end{proof}

\begin{definition} An $\mathcal{NAS}_k(n)$ ($n\geq2$, $k\geq3$) is said to be \emph{maximal} if
it has period meeting the bound of Theorem~\ref{theorem:NAS_bound}.
\end{definition}

In the remainder of this paper we are concerned with determining the existence of maximal negative
avoiding sequences for all possible $n\geq2$ and $k\geq3$. We give three examples showing that such
sequences exist for small values of $n$ and $k$.

\begin{example}  \label{example:small_examples}
If $k=3$ and $n=2$, then $[0,1,1,2]$ is a maximal $\mathcal{NAS}_3(2)$, since it has period
$(3^2-1)/2=4$.

If $k=4$ and $n=2$, then $[1,3,0,1,2,1]$ is a maximal $\mathcal{NAS}_4(2)$, since it has period
$(4^2-2^2)/2=6$.

If $k=3$ and $n=3$ then $[0,1,0,0,1,1,1,0,1,2,1,1,2]$ is a maximal $\mathcal{NAS}_3(3)$, since it
has period $(3^3-1)/2=13$.
\end{example}

Before proceeding observe that, if $k\geq3$ and $n\geq2$, it follows immediately from the
definition that if $[a_0,a_1,\dots,a_{m-1}]$ is a $\mathcal{NAS}_k(n)$, then so is
$[-a_0,-a_1,\dots,-a_{m-1}]$. Hence if $[a_0,a_1,\dots,a_{m-1}]$ is a maximal $\mathcal{NAS}_k(n)$,
then $[a_0,a_1,\dots,a_{m-1}]$ and $[-a_0,-a_1,\dots,-a_{m-1}]$ contain all the non-self-negative
$n$-tuples.  The self-negative tuples can be arranged to form a sequence of period 1 or $2^n$,
depending on whether $k$ is odd or even. These three sequences can be merged to generate a de
Bruijn sequence since the de Bruijn graph is strongly connected.

\section{Maximal sequences for odd \texorpdfstring{$n$}{n}}    \label{section:maximal_oddn}

We next establish the existence of a maximal $\mathcal{NAS}_k(n)$ for any odd $n\geq3$ and any
$k\geq 3$.

\subsection{Preliminaries}

We first need the following concept.

\begin{definition}[pseudoweight, \cite{Mitchell25a}]  \label{definition:pseudoweight}
Suppose $\mathbf{u}=(u_0,u_1,\dots,u_{n-1})$ is an $n$-tuple of elements of $\mathbb{Z}_k$
($n\geq1$, $k\geq3$). Define the function $f:\mathbb{Z}_k\rightarrow\mathbb{Q}$ as follows: for
any $u\in\mathbb{Z}_k$ treat $u$ as an integer in the range $[0,k-1]$ and set $f(u)=u$ if
$u\neq0$ and $f(u)=k/2$ if $u=0$. Then the \emph{pseudoweight} of $\mathbf{u}$ is defined to be
the sum
\[ w^*(\mathbf{u}) = \sum_{i=0}^{n-1}f(u_i) \]
where the sum is computed in $\mathbb{Q}$.
\end{definition}

As a simple example for $k=3$, the 4-tuple $(0,1,1,2)$ has weight $0+1+1+2=4$ and pseudoweight
$1.5+1+1+2=5.5$, since $f(0)=\frac{3}{2}$.

\begin{definition}[\cite{Mitchell26}]
Suppose $n\geq2$ and $k\geq3$. Let $E_k(n)$ be the set of all $k$-ary $n$-tuples with pseudoweight
less than $nk/2$.
\end{definition}

We can now give a key result.

\begin{theorem}[\cite{Mitchell26}]  \label{theorem:pseudoweight}
Suppose $n\geq2$ and $k\geq3$. $E_k(n)$ is the set of edges for an Eulerian subgraph $U_k(n-1)$, of
the de Bruijn digraph $B_{k}(n-1)$.
\end{theorem}

We also have the following simple result, which follows immediately given that, for any $k$-ary
$n$-tuple $\mathbf{a}$, $w^*(\mathbf{a})=nk-w^*(-\mathbf{a})$.

\begin{lemma}  \label{lemma:Ukn_antinegative}
Suppose $n\geq2$ and $k\geq3$. $U_k(n-1)$ is an antinegative subgraph of the de Bruijn digraph
$B_{k}(n-1)$.
\end{lemma}

\begin{example}[cf.\ Example 5.2 of \cite{Mitchell26}]  \label{example:E33}
As an example, consider the case $k=3$ and $n=3$. The ten 3-ary 3-tuples having pseudoweight less
than 4.5 are
\[ 111,~011,~101,~110,~001,~010,~100,~112,~121,~211, \]
forming the set $E_3(3)$.  These can readily be arranged to form a $\mathcal{NAS}_3(3)$ of period
10 (less than the maximum possible value which is 13), e.g.:
\[  [1110010112]. \]
\end{example}

While the above analysis immediately shows how to construct an antinegative Eulerian subgraph
for any $k$ and $n$ (and hence a $\mathcal{NAS}_k(n)$ from
Theorem~\ref{theorem:correspondence}), the sequences will not be maximal, since they do not
contain any $n$-tuples with weight exactly $kn/2$.  We next show how, if $n$ is odd, we can
rectify this.  We first need to introduce a little more notation.

\begin{definition}[\cite{Mitchell26a}]
Suppose $n\geq2$ and $k\geq3$. Let $H_k(n-1)$ be the subgraph of $B_k(n-1)$ restricted to edges of
pseudoweight precisely $kn/2$.
\end{definition}

\subsection{The construction}

Following \cite{Mitchell26a}, we present a simple way of dividing the edges ($n$-tuples) in
$H_k(n-1)$ into edge-disjoint circuits. Let $(a_0,a_1,\dots,a_{n-1})$ be a $k$-ary $n$-tuple, i.e.\
an edge in $B_k(n-1)$. Let $m$ be the smallest positive integer $c$ such that
$a_i=a_{\overline{i+c}}$ for every $i$ ($0\leq i<n$), where $\overline{x}=x\bmod n$. We write
$[a_0,a_1,\dots,a_{n-1}]$ for the circuit in $B_k(n-1)$ consisting of edges
\[ (a_0,a_1,\dots,a_{n-1}), (a_1,a_2\dots,a_{n-1},a_0), \dots,
(a_{m-1},\dots,a_{n-1},a_0,a_1,\dots,a_{m-2}), \] and say the circuit has period $m$.

\begin{lemma}[\cite{Mitchell26a}]  \label{lemma:k_odd_partitioning edges}  \label{lemma:defining_circuits}
Suppose $k\geq3$ and $n\geq 3$.
\begin{enumerate}
\item[i)] If $(a_0,a_1,\dots,a_{n-1})$ is a $k$-ary $n$-tuple of pseudoweight $kn/2$, i.e.\ an
    edge in $H_k(n-1)$, then $[a_0,a_1,\dots,a_{n-1}]$ is a circuit in $H_k(n-1)$ of
     period dividing $n$.
\item[ii)] Let~$\mathcal{C}_k(n-1)$ be the set of all circuits of the form
    $[a_0,a_1,\dots,a_{n-1}]$, where $(a_0,a_1,\dots,a_{n-1})$ is an edge in $H_k(n-1)$; then
    $\mathcal{C}_k(n-1)$ forms an edge-disjoint partition of the edges in $H_k(n-1)$.
\end{enumerate}
\end{lemma}

Observe that if a circuit contains a self-negative $n$-tuple, then all the $n$-tuples
in the circuit are self-negative, and we refer to the circuit as being self-negative.

\begin{definition}
Let $\mathcal{C}^*_k(n-1)$ be the set of non-self-negative circuits in $\mathcal{C}_k(n-1)$.
\end{definition}

The following observation is key.

\begin{lemma}
Suppose $k\geq3$ and $n\geq 3$ is odd.  No circuit in $\mathcal{C}^*_k(n-1)$ can contain both an
$n$-tuple and its negative.
\end{lemma}

\begin{proof}
Suppose not, i.e.\ suppose
\[ (a_0,a_1,\dots,a_{n-1})=-(a_{\overline{j}},a_{\overline{j+1}},\dots,a_{\overline{j-1}}) \]
where $\overline{j}=j\bmod n$.  Hence
\[ a_i=(-1)^s a_{\overline{sj+i}} \]
for every $i$ ($0\leq i<n$) and every $s\geq0$.  Setting $s=n$ and observing that $n$ is odd
immediately gives
\[ a_i=- a_i \]
for every $i$ ($0\leq i<n$), and hence $(a_0,a_1,\dots,a_{n-1})$ is self-negative, giving a
contradiction.
\end{proof}

This immediately gives the following.

\begin{lemma}  \label{lemma:C*_pairs}
Suppose $k\geq3$ and $n\geq 3$ is odd.  The circuits in $\mathcal{C}^*_k(n-1)$ can be divided into
pairs, consisting of a circuit and its negative.
\end{lemma}

We can now describe our simple construction.

\begin{construction}  \label{the construction:n_odd}
Suppose $k\geq3$ and $n\geq 3$ is odd.  Select one of each pair of circuits in
$\mathcal{C}^*_k(n-1)$, as per Lemma~\ref{lemma:C*_pairs}, and let $F_k(n)$ be the set of $k$-ary
$n$-tuples containing all the $n$-tuples in $E_k(n)$ together with all the $n$-tuples occurring in
the selected circuits from $\mathcal{C}^*_k(n-1)$. Let $W_k(n-1)$ be the subgraph of $B_k(n-1)$ whose edges correspond to $n$-tuples in $F_k(n)$ and whose vertices belong to these edges.
\end{construction}

\begin{lemma}  \label{construction:n_odd}
Suppose $k\geq3$ and $n\geq 3$ is odd.   Then $W_k(n-1)$ is an
antinegative Eulerian subgraph of $B_k(n-1)$, and $|F_k(n)|=\frac{k^n-\delta^n}{2}$, where
$\delta=1$ or $2$ depending on whether $k$ is odd or even.
\end{lemma}

\begin{proof}
From Theorem~\ref{theorem:pseudoweight} and Lemma~\ref{lemma:Ukn_antinegative} we know that
$E_k(n)$ is the set of edges for antinegative Eulerian subgraph of $B_k(n-1)$. All the edges in
$E_k(n)$ have pseudoweight less than $kn/2$, and hence since they only contain edges of weight
exactly $kn/2$, neither an $n$-tuple nor its negative from a circuit in $\mathcal{C}^*_k(n-1)$ can
equal an $n$-tuple in $E_k(n)$; moreover, since we only selected one circuit of each pair from
$\mathcal{C}^*_k(n-1)$, it follows that $F_k(n)$ is the set of edges for an antinegative subgraph
of $B_k(n-1)$.

If $r$ is the number of $k$-ary $n$-tuples of pseudoweight exactly $kn/2$, then clearly
$|E_k(n)|=\frac{k^n-r}{2}$.  Moreover, since the circuits in $\mathcal{C}^*_k(n-1)$ cover all the
non-self-negative $n$-tuples of pseudoweight exactly $kn/2$, the number of $n$-tuples occurring in
a circuit in $\mathcal{C}^*_k(n-1)$ is $r-\delta^n$.  Since $F_k(n)$ contains precisely half the
$n$-tuples occurring in a circuit in $\mathcal{C}^*_k(n-1)$,
\[ |F_k(n)| = |E_k(n)| + \frac{r-\delta^n}{2} = \frac{k^n-r}{2} + \frac{r-\delta^n}{2} =
   \frac{k^n-\delta^n}{2}. \]

It remains to show that the subgraph is Eulerian, i.e.\ it is balanced and connected.  Since
$U_k(n-1)$ is Eulerian (from Theorem~\ref{theorem:pseudoweight}), and since we are adding complete
circuits to $E_k(n)$, it follows that the subgraph is balanced. Moreover, every circuit in
$\mathcal{C}^*_k(n-1)$ must contain an edge whose corresponding
$n$-tuple contains at least one element in the range $[k/2,k-1]$, and hence it must
contain an edge on a vertex whose corresponding $(n-1)$-tuple has pseudoweight at most $(n-1)k/2$.
Now every vertex whose corresponding $(n-1)$-tuple has pseudoweight at most $(n-1)k/2$ has ingoing
and outgoing edges in $E_k(n)$, and hence the subgraph with edge set $F_k(n)$ is connected.
\end{proof}

Combining Lemma~\ref{construction:n_odd} with Theorems~\ref{theorem:correspondence} and
\ref{theorem:NAS_bound} immediately gives the following.

\begin{theorem}  \label{theorem:maximal_NAS_n_odd}
If $k\geq3$ and $n\geq 3$ is odd then there exists a maximal $\mathcal{NAS}_k(n)$.
\end{theorem}

\begin{example}[cf.\ Example 2.1 of \cite{Mitchell26a}]  \label{example:n3k3}
Suppose $n=k=3$.  The seven edges in $H_3(2)$ can be partitioned into three circuits:
\[ [000], [012], [021] \]
of periods 1, 3 and 3, respectively.  Thus $\mathcal{C}^*_3(2)=\{[012], [021]\}$, i.e.\ it contains
one pair of a circuit and its negative.  Choosing one, e.g., $[012]$, and adding it to the
$\mathcal{NAS}_3(3)$ of Example~\ref{example:E33}, i.e.\ $[1110010112]$, we obtain the following
maximal $\mathcal{NAS}_3(3)$ of period 13:
 \[ [1110012010112]. \]
\end{example}

\subsection{The \texorpdfstring{$n$}{n} even case is not so simple}

Things are less straightforward if $n$ is even, since in this case a circuit in
$\mathcal{C}^*_k(n-1)$ \emph{can} contain both an $n$-tuple and its negative. Indeed, the following
simple example means that, for $n$ even, we cannot create maximal NASs simply by taking all tuples
of pseudoweight less than $nk/2$ and then adding in selected tuples of weight exactly $nk/2$.

\begin{example}
Suppose $k=3$ and $n=4$. If a maximal $\mathcal{NAS}_3(4)$ exists, then it cannot only contain
$4$-tuples with pseudoweight less than or equal to $kn/2=6$.  To see why this is true, suppose that
$S$ is a maximal $\mathcal{NAS}_3(4)$ (i.e.\ one with period 40) containing only 4-tuples with
pseudoweight less than or equal to $6$. It must contain all 31 4-tuples of pseudoweight less than
6, and half the 18 4-tuples with pseudoweight exactly 6.

Next observe that $S$ cannot contain a 3-tuple of the form $(2,1,2)$. If $S$ did contain an
occurrence of $(2,1,2)$ then such a 3-tuple must be preceded and succeeded by 1 within $S$, given
that $S$ only contains 4-tuples of pseudoweight at most 6.  However, if $S$ contains the 5-tuple
$(1,2,1,2,1)$ then $S$ will contain the 4-tuples $(1,2,1,2)$ and $(2,1,2,1)$, which gives a
contradiction since they are the negative of each other.  Hence $S$ cannot contain any 3-tuples of
the form $(2,1,2)$.  Now consider the 4-tuples $(1,2,1,2)$ and $(2,1,2,1)$ of weight 6. $S$ must
contain one of them --- which immediately gives a contradiction since they both contain the 3-tuple
$(2,1,2)$.
\end{example}

\section{Maximal sequences for \texorpdfstring{$n=2$}{n=2}}  \label{section:maximal_n2}

As a first step towards demonstrating the existence of maximal negative avoiding sequences for all
even $n$, we start by looking at $n=2$.  We consider the $k$ odd and even cases separately.

\subsection{Maximal sequences for \texorpdfstring{$n=2$}{n=2} and \texorpdfstring{$k$}{k} odd}

\begin{definition}
Suppose $k\geq3$ is odd.  Let $Y_k\subseteq \mathbb{Z}_k^2$ be the following set of $k$-ary
2-tuples:
\[ \{ (x,y): \overline{y-x}\in \{1,2,\dots,(k-1)/2\} \} \cup \{ (x,x): x\in \{1,2,\dots,(k-1)/2\} \}, \]
where $\overline{m}$ denotes $m\bmod k$.
\end{definition}

\begin{lemma}  \label{lemma:construction_n2}
Suppose $k\geq3$ is odd.  Then:
\begin{itemize}
\item[(i)] $|Y_k|=\frac{k^2-1}{2}$;
\item[(ii)] the 2-tuples in $Y_k$ correspond to the set of edges of an Eulerian subgraph of
    $B_k(1)$;
\item[(iii)] if $(x,y)\in Y_k$ then $(\overline{-x},\overline{-y})\not\in Y_k$.
\end{itemize}
\end{lemma}

\begin{proof}
For each of the $k$ possible values for $x$, there are precisely $(k-1)/2$ values $y$ for which
$\overline{y-x}\in \{1,2,\dots,(k-1)/2\}$. Hence
$|Y_k|=\frac{k(k-1)}{2}+\frac{k-1}{2}=\frac{k^2-1}{2}$, and (i) follows.

Consider any vertex $(x)$ in $B_k(1)$. If $x\in \{1,2,\dots,(k-1)/2\}$ then $(x)$ has in-degree and
out-degree $(k+1)/2$ in the subgraph defined by $Y_k$; otherwise $(x)$ has in-degree and out-degree
$(k-1)/2$. That is every vertex has its in-degree equal to its out-degree. To show it is Eulerian
it remains to show that it is connected.  Suppose $(x)$ and $(y)$ are distinct vertices.  If
$\overline{y-x}\in \{1,2,\dots,(k-1)/2\}$ then $(x,y)\in Y_k$; otherwise $(y,x)\in Y_k$, and hence
(ii) follows.

Finally, if $x\neq y$ then $(x,y)\in Y_k$ implies $\overline{y-x}\in \{1,2,\dots,(k-1)/2\}$, and
hence $\overline{x-y}\not\in \{1,2,\dots,(k-1)/2\}$, i.e.\ $(\overline{-x},\overline{-y})\not\in
Y_k$.  If $x=y$ then $(x,x)\in Y_k$ implies $x\in \{1,2,\dots,(k-1)/2\}$, and hence
$\overline{-x}\not\in \{1,2,\dots,(k-1)/2\}$, i.e.\ $(\overline{-x},\overline{-x})\not\in Y_k$, and
(iii) follows.
\end{proof}

Since, from Lemma~\ref{lemma:construction_n2}(ii), the 2-tuples in $Y_k$ correspond to the set of
edges of an Eulerian subgraph of $B_k(1)$, it is possible to define a sequence $S$ in which each of
the 2-tuples in $Y_k$ occurs exactly once. Such a sequence is a $\mathcal{NAS}_k(2)$ from
Lemma~\ref{lemma:construction_n2}(iii).  Moreover, from Lemma~\ref{lemma:construction_n2}(i), $S$
has period $|Y_k|=\frac{k^2-1}{2}$. This immediately yields the following from
Theorems~\ref{theorem:correspondence} and \ref{theorem:NAS_bound}.

\begin{theorem}  \label{theorem:maximal_NAS_n2_kodd}
If $k\geq 3$ is odd there exists a maximal $\mathcal{NAS}_k(2)$.
\end{theorem}

\begin{example}
Suppose $k=3$.  Then
\[ Y_3 = \{ (0,1),~(1,2),~(2,0),~(1,1) \}. \]
These can be joined to create a maximal $\mathcal{NAS}_3(2)$ of period 4, namely: $[0,1,1,2]$ (cf.\
Example~\ref{example:small_examples}).
\end{example}

\subsection{Maximal sequences for \texorpdfstring{$n=2$}{n=2} and \texorpdfstring{$k$}{k} even}

\begin{construction}  \label{construction_n2_even}
Suppose $k\geq4$ is even. Observe that $U_k(1)$, the Eulerian subgraph of $B_k(1)$ with edge set
$E_k(2)$, contains (amongst others) the edges $(0,i),(i,i)$ for every $i$, $0<i<k/2$. Let
$U_k^{\prime}(1)$ be the subgraph of $B_k(1)$ obtained by removing these edges from $U_k(1)$ and
adjoining the edges $(0,-i),(-i,-i),(-i,i)$ for every $i$, $0<i<k/2$.
\end{construction}

\begin{lemma}
If $k\geq4$ is even then $U_k^{\prime}(1)$ is an antinegative Eulerian subgraph of $B_k(1)$
containing $(k^2-4)/2$ edges.
\end{lemma}

\begin{proof}
Making the specified modifications to $U_k(1)$ maintains the number of in- and out-edges to the
vertex with label $0$, reduces the number of in- and out- edges to vertex $i$, $0<i<k/2$, by $1$,
but not to zero, and increases the number of in- and out- edges to vertex $-i$, $0<i<k/2$, by $2$.
The subgraph $U_k^{\prime}(1)$ is therefore Eulerian as now vertex $k-1$ is connected to vertices
$0$ and $1$, and vertex $1$ remains connected to all other vertices. It is antinegative as the
adjoined edges are not negatives of the existing edges or of each other. Since two edges have been
removed and three edges adjoined for each $i$, $0<i<k/2$, $U_k^{\prime}(1)$ has $(k^2-k-2)/2 +
(k-2)/2 = (k^2-4)/2$ edges.
\end{proof}

We then immediately have the following result from Theorems~\ref{theorem:correspondence} and
\ref{theorem:NAS_bound}.

\begin{theorem}  \label{theorem:maximal_NAS_n2_keven}
If $k\geq 4$ is even there exists a maximal $\mathcal{NAS}_k(2)$.
\end{theorem}

\begin{example}
Suppose $k=4$.  Then $U_4(1)$ contains five edges with labels:
\[ \{ (0,1), (1,0), (1,1), (1,2), (2,1)  \}. \]
Creating $U_4^{\prime}(1)$ involves removing edges $(0,1)$ and $(1,1)$ and inserting edges
$(0,3)$, $(3,3)$ and $(3,1)$, i.e., $U_4^{\prime}(1)$ contains the following edges:
\[  \{ (0,3), (1,0), (3,3), (3,1), (1,2), (2,1)  \}. \]
These yield a maximal $\mathcal{NAS}_4(2)$ of period 6, e.g., $[0,3,3,1,2,1]$.
\end{example}

\section{Maximal sequences for even \texorpdfstring{$n\geq 4$}{n>=4}}  \label{section:maximal_neven}

To construct maximal negative avoiding sequences of even order  we apply the pre-image mapping
$D^{-1}$ to the subgraph $W_k(n-1)$ of Construction~\ref{the construction:n_odd} for $n$ odd and
adjoin to the resulting set of pre-image edges certain additional edges, as described below. We
begin with a lemma about $D^{-1}(W_k(n-1))$.

\begin{lemma}  \label{lemma:Wkn_properties even}
Suppose $n\geq3$ and $k\geq3$, where $n$ is odd. The subgraph
$D^{-1}(W_k(n-1))$ of $B_k(n)$ has the following properties:
\begin{itemize}
\item[(i)] The number of edges of $D^{-1}(W_k(n-1))$ is $\frac{k^{n+1}-k\delta^n}{2}$;
\item[(ii)] $D^{-1}(W_k(n-1))$ is antinegative;
\item[(iii)] $D^{-1}(W_k(n-1))$ is an Eulerian subgraph of the de Bruijn graph $B_k(n)$.
\end{itemize}
\end{lemma}

\begin{proof}
Property (i) follows from Lemma~\ref{lempel inverse size}. Property (ii) follows because
$D^{-1}$ preserves antinegativity. $W_k(n-1)$ is balanced so $D^{-1}(W_k(n-1))$ has this
property too. By Theorem 5.5 of \cite{Mitchell26} $D^{-1}(U_k(n-1))$ is a connected subgraph of
the de Bruijn graph $B_k(n)$. Since each circuit of $\mathcal{C}^*_k(n-1)$ shares a vertex with
$U_k(n-1)$ it follows that $D^{-1}(W_k(n-1))$ is also connected. Hence (iii) holds.
\end{proof}

Let $T_k(n-1)$ be the subgraph of $B_k(n-1)$ whose edges correspond to all the self-negative
$n$-tuples, i.e., edges corresponding to $n$-tuples with all entries equal to $0$ or $k/2$, and
whose vertices are those on such edges. When $k$ is odd, $T_k(n-1)$ consists of one edge on one
vertex (a loop). When $k$ is even $T_k(n-1)$ consists of $2^n$ edges on $2^{n-1}$ vertices and is
isomorphic to $B_2(n-1)$. So $T_k(n-1)$ is Eulerian. Let $\mathbf{a}$ be an Eulerian circuit in
$T_k(n-1)$. Then $\mathbf{a}$ corresponds to a sequence $[a_0,a_1,\dots,a_{m-1}]$ where $m=1$ if
$n$ is odd and $m=2^n$ if $n$ is even, and $a_0=a_1=\dots =a_{n-1}=0$. Now, since
$b_{m-1}=\sum_{i=0}^{m-1}a_i=0$, as $k/2$ occurs an even number of times, $D^{-1}(\mathbf a)$ can
be partitioned into $k$ circuits $\mathbf{b}_x$, $x \in \mathbb{Z}_k$, of $B_k(n)$ corresponding to
sequences $[x,x+a_0,x+a_0+a_1,\dots,x+\sum_{i=0}^{m-2}a_i]=[x,x+b_0,x+b_1,\dots,x+b_{m-2}]$, where
$b_i=\sum_{j=0}^{i}a_j$, $i=0,\dots,m-2$.

%Let $\mathbf a=[a_0,a_1,\dots,a_{m-1}]$ be the sequence corresponding to an Eulerian circuit
%$A$ in $T_k(n-1)$, so that $A=E_n(a)$, where $m=n$ if $k$ is odd (and $a_i=0, i=0,1,\dots,n-1$)
%and $m=2^n$ if $k$ is even. Now, since $\sum_{i=0}^{m-1} a_i=0$ as $k/2$ occurs an even number
%of times, $D^{-1}(\mathbf a)$ can be partitioned into $k$ circuits $A_i$ ($0\leq i<k$) of
%$B_k(n)$, see Lemma~\ref{lempel inverse size}, which may be written as
%$\mathbf{b}_x=E_n(A_x)=[x,x,\dots,x]$ (of length $n+1$) if $k$ is odd, and
%$\mathbf{b}_x=E_n(A_x)=[b_0,b_1,\dots,b_{2^n-1}]=[x,x+a_0,x+a_0+a_1,\dots,x+a_0+a_1+\dots+a_{2^n-2}]$
%(of length $2^n$), i.e., $b_i=x+\sum_{j=0}^{i-1} a_j$, if $k$ is even, where
%$x\in\mathbb{Z}_k$.

Let $Z_k(n)$ be the subgraph of $B_k(n)$ whose edges are the edges of $D^{-1}(W_k(n-1))$ and
the edges of $ \mathbf{b}_x$ for $0<x<k/2$.

\begin{theorem}  \label{theorem:maximal_NAS_neven_kodd}
Suppose $k\geq3$ and $n\geq3$ is odd. Then $Z_k(n)$ is an Eulerian subgraph of $B_k(n)$ and there
exists a maximal $\mathcal{NAS}_k(n+1)$.
\end{theorem}

\begin{proof}
Each edge of $D^{-1}(W_k(n-1))$ has as its image under $D$ an edge which is not self-negative,
and each edge of $\mathbf{b}_x$ has an image which is self-negative. So $D^{-1}(W_k(n-1))$ and
$\mathbf{b}_x$ are edge-disjoint for every $x$. Moreover, if $x_1\neq x_2$ then
$\mathbf{b}_{x_1}$ and $\mathbf{b}_{x_2}$ do not share an edge, since if
%$(\mathbf{b}_{x_1})_{n+1}(i)=(\mathbf{b}_{x_2})_{n+1}(j)$
$(x_1+b_i,x_1+b_{i+1},\dots,x_1+b_{i+n})=(x_2+b_j,x_2+b_{j+1},\dots,x_2+b_{j+n})$ then either
$i=j$ and $x_1=x_2$ (a contradiction) or $i\neq j$ and
%$\mathbf{a}_n(i)=\mathbf{a}_n(j)$
$(a_i,a_{i+1},\dots,a_{i+n-1})=(a_j,a_{j+1},\dots,a_{j+n-1})$ (a contradiction). By
Lemma~\ref{lemma:Wkn_properties even}(i) the number of edges of $Z_k(n)$ is
\[ \frac{k^{n+1}-k\delta^n}{2}+\frac{k-\delta}{2}\delta^n=\frac{k^{n+1}-\delta^{n+1}}{2},\]
where, as before, $\delta=1$ if $k$ is odd and $\delta=2$ if $k$ is even, since each circuit
$\mathbf{b}_x$ contains $\delta^n$ edges.

By Lemma~\ref{lemma:Wkn_properties even}(iii) and the fact that each circuit $\mathbf{b}_x$ is
Eulerian, $\mathbb Z_k(n-1)$ is Eulerian since $D^{-1}(W_k(n-1))$ shares a vertex with each
$\mathbf b_x$, e.g., $(x,x,\dots,x)$ which is a vertex of both $\mathbf b_x$ and
$D^{-1}(W_k(n-1))$, as $(x,0,\dots,0,)$, having pseudoweight less than $nk/2$, belongs to
$W_k(n-1)$.

%and shares a vertex with $D^{-1}(W_k(n-1))$,
%namely $(x,x,\dots,x)$, which is a vertex of the pre-image (with $y=0$) of the edge
%$(x,0,0,\dots,0)$ (of length $n$) which, since it has pseudoweight less than $nk/2$, is an edge
%of $W_k(n-1)$, and hence $Z_k(n)$ is connected.  $Z_k(n)$ is also balanced since
%$D^{-1}(W_k(n-1))$ is balanced (by Lemma~\ref{lemma:Wkn_properties even}(iii)) and each
%$\mathbf{b} x$ is itself a closed circuit and hence trivially balanced.

Now $D^{-1}(W_k(n-1))$ is antinegative from Lemma~\ref{lemma:Wkn_properties even}(ii). Since
%$-\mathbf{a}= \mathbf{a}$,
$-a_i=a_i$, $i=0,1,\dots,m-1$, it follows that
%$-\mathbf{b}_x = \mathbf{b}_{-x}$
$-(x+b_i)=-x+b_i$, $i=0,1,\dots,m-1$, and as $k/2<-x<k$ when $0<x<k/2$ the negative of an edge
of any $\mathbf{b}_x$, $0<x<k/2$, does not belong to $Z_k(n)$. Thus $Z_k(n)$ is antinegative.
Hence there exists an Eulerian circuit in $Z_k(n)$, which corresponds to an
$\mathcal{NAS}_k(n+1)$ of period $\frac{k^{n+1}-\delta^n}{2}$, i.e., a maximal
$\mathcal{NAS}_k({n+1})$.
\end{proof}

The above theorem completes the project of demonstrating the existence of a maximal
$\mathcal{NAS}_k(n)$ for every $n\geq2$ and every $k\geq3$.

\begin{example}
As an example of the above construction, suppose $k=n=3$.  From Examples~\ref{example:E33} and
\ref{example:n3k3}, we can set $W_3(2)$ to be the subgraph containing the 13 edges corresponding to
the following $3$-tuples:
\[ 111,  011,  101,  110,  001,  010,  100,  112,  121,  211,  012,  120,  201. \]
Then $D^{-1}(W_3(2))$ contains the following 39 $4$-tuples:
\begin{align*}
0120, 0012, 0112, 0122, 0001, 0011, 0111, 0121, 0101, 0201, 0010, 0100, 0220, \\
1201, 1120, 1220, 1200, 1112, 1122, 1222, 1202, 1212, 1012, 1121, 1211, 1001, \\
2012, 2201, 2001, 2011, 2220, 2200, 2000, 2010, 2020, 2120, 2202, 2022, 2112.
\end{align*}
$Z_3(3)$ contains one additional edge, namely $\mathbf{a}_1=1111$, giving a set of 40 4-tuples
that can be joined to create a maximal $\mathcal{NAS}_3(4)$, e.g.:
\[  [0001012120220201001201220112112001111222] \].
\end{example}

\begin{example}
As a second example of the above construction, suppose $k=4$ and $n=3$.

The set $E_4(3)$ of $3$-tuples over $\mathbb{Z}_4$ of pseudoweight less than $nk/2=6$ has
$|E_4(3)|=22$ elements, as follows:
\[
\begin{aligned}
E_4(3)=\{\,&001,010,011,012,021,100,101,102,110,111,\\
&112,113,120,121,122,131,201,210,211,212,221,311\,\}.
\end{aligned}
\]

The set $H_4(2)$ of $3$-tuples of pseudoweight exactly $nk/2=6$ has $|H_4(2)|=20$ elements,
which Lemma~5.3 partitions into eight circuits:
\[
[000],\quad
[002]=\{002,020,200\},\quad
[022]=\{022,202,220\},\quad
[222],
\]
all four of which are self-negative, and
\[\begin{aligned}
&[013]=\{013,130,301\},\quad
[031]=\{031,103,310\},\\
&[123]=\{123,231,312\},\quad
[132]=\{132,213,321\},
\end{aligned}\]
which are not. Thus
\[
C_4^{*}(2)=\{[013],[031],[123],[132]\},
\]
and under negation we obtain the pairs $\{[013],[031]\}$ and $\{[123],[132]\}$. Selecting one
circuit from each pair, say $[013]$ and $[123]$, adjoin the six $3$-tuples
$013,130,301,123,231,312$.

Set $W_4(2)$ to be the set of edges in
\[
F_4(3)=E_4(3)\cup\{013,130,301\}\cup\{123,231,312\},
\]
so $|F_4(3)|=22+3+3=28=(4^3-2^3)/2$, as required by Lemma~\ref{construction:n_odd}. Then
$W_4(2)$ is antinegative and balanced, and is hence an antinegative Eulerian subgraph of
$B_4(2)$ (and incidentally yields a maximal $NAS_4(3)$ of period $28$).

For each edge $e=(e_0,e_1,e_2)\in W_4(2)$ and each $x\in\mathbb{Z}_4$ we obtain a $4$-tuple
\[
\bigl(x,\;x+e_0,\;x+e_0+e_1,\;x+e_0+e_1+e_2\bigr).
\]
This produces $D^{-1}(W_4(2))$, with
\[
|D^{-1}(W_4(2))|=4\times 28=112=\frac{4^{4}-4\cdot 2^{3}}{2},
\]
as per Lemma~\ref{lemma:Wkn_properties even}(i); it is antinegative and balanced by
Lemma~\ref{lemma:Wkn_properties even}(ii),(iii).

$T_4(2)$ consists of the eight self-negative $3$-tuples over $\mathbb{Z}_4$ (with entries in
$\{0,2\}$) and is isomorphic to $B_2(2)$. Taking the binary de Bruijn sequence
$[0,0,0,1,0,1,1,1]$ in $B_2(2)$ and mapping $0\mapsto 0$, $1\mapsto 2$, gives the Eulerian
circuit
\[
%a=[0,0,0,2,0,2,2,2],
\mathbf{a}=E_n([0,0,0,2,0,2,2,2]).
\]
where the elements sum to $0$ in $\mathbb{Z}_4$.

Since $k/2=2$, we adjoin $\mathbf{b}_x$ only for $x=1$. Setting $\mathbf{b}_1 =
E_{n+1}([1,1,1,1,3,3,1,3])$
%whose eight cyclic windows of length 4 give the additional edges
gives eight additional edges
\[
A = \{ 1111,~1113,~1133,~1331,~3313~3131,~1311,~3111 \}.
\]

Combining, we get
\[
Z_4(3)=D^{-1}(W_4(2))~\cup~A\},
\]
where
\[
|Z_4(3)| = 112+8 = 120 = \frac{4^{4}-2^{4}}{2},
\]
and $Z_4(3)$ is antinegative, balanced, and contains no self-negative $4$-tuple. An example of
an Eulerian circuit through the $120$ edges of $Z_4(3)$ is the following maximal $NAS_4(4)$ of
period $120$:
\[\begin{aligned}
&[0,0,0,1,3,3,1,3,1,2,3,3,3,0,3,3,0,2,3,3,1,2,2,3,3,0,1,3,2,3,\\
& 1,3,0,2,2,3,2,3,0,3,0,2,1,2,3,2,2,3,1,2,0,3,0,1,2,3,1,1,2,3,\\
& 0,2,0,1,3,1,1,1,3,3,0,0,2,3,1,0,1,3,0,1,1,3,0,0,1,2,2,2,3,0,\\
& 1,0,1,2,1,2,0,1,2,0,0,1,0,0,1,1,2,2,0,1,1,1,1,2,1,1,2,0,2,3].
\end{aligned}\]
\end{example}

\section{Conclusions and open questions}    \label{section:conclusions}

We have introduced a special category of cut-down de Bruijn sequences, the negative
avoiding sequences, and gave an upper bound on their period.  We then showed by construction that
the bound is tight for all $n\geq2$ and $k\geq3$.

The sequences under study are an example of a special category of cut-down de Bruijn sequences, as
are also the orientable and negative orientable sequences.  For $n\geq2$ and $k\geq2$, these
sequences can be characterised as follows.

\begin{definition}
Suppose $p$ is a self-inverse permutation (i.e., satisfying $p^{-1}=p$) on the set of $k$-ary
$n$-tuples, i.e., on the set $(\mathbb{Z}_k)^n$. Then a \emph{pair-selecting sequence} for the
permutation $p$ is a $k$-ary cut-down de Bruijn sequence of span $n$ with the property that if the
$n$-tuple $\mathbf{a}$ occurs in the sequence then $p(\mathbf{a})$ does not.
\end{definition}

It is simple to see that orientable, negative orientable and negative avoiding sequences are all
examples of pair-selecting sequences for appropriate choices of $p$:
\begin{itemize}
\item for a negative avoiding sequence, $p$ is defined so that $p(\mathbf{a})=-\mathbf{a}$;
\item for an orientable sequence, $p$ is defined so that $p(\mathbf{a})=\mathbf{a}^R$;
\item for a negative orientable sequence, $p$ is defined so that $p(\mathbf{a})=-\mathbf{a}^R$;
\end{itemize}
where for any $n$-tuple $\mathbf{a}$, $\mathbf{a}^R$ denotes the $n$-tuple obtained by reversing
the order of its entries.  It should thus be clear that these three categories of sequence are
closely related in some sense, although, unlike the case for negative avoiding sequences, the
problem of determining the maximum possible period of orientable and negative orientable sequences
remains unresolved for all but small $n$ (see, for example, \cite{Mitchell26a}).

More generally, one might ask whether it is possible to obtain more general results on such
sequences applying for large classes of permutation $p$.  One obvious result is that the upper
bound on the period of such a sequence will always be $\frac{k^n-f^n}{2}$, where $f$ is the number
of fixed points of the permutation $p$.  Of course, this bound will not always be sharp, as is
demonstrated by orientable and negative orientable sequences for which smaller upper bounds can be
derived.

A special case of this question would be where the permutation $p$ derives from a
self-inverse permutation of the elements of $\mathbb{Z}_k$, as is the case
for negative avoiding sequences. It would interesting to know whether the trivial upper bound on
the period is always achievable in this case.

The simplest case of this is when $k=2$ and $p$ simply switches 0 to 1 and vice versa; in this case
there are no fixed points, i.e., $f=0$, and hence a maximal sequence has period $2^{n-1}$. If $S$
is such a maximal sequence, then $p(S)$, i.e., $S$ with every bit switched, is clearly also a
maximal sequence, and between them $S$ and $p(S)$ contain every possible $n$-tuple exactly once in
a period.  We refer to such a pair of maximal sequences as a \emph{complementary pair}.

If $n$ is odd then a simple construction involves the set of binary $n$-tuples of Hamming
weight less than $n/2$. That this set defines an Eulerian subgraph of $B_2(n-1)$ was
established by Sawada at al.~\cite{Sawada11}, as observed in Section~\ref{section:motivation}.
The size of the set is clearly $2^{n-1}$ and hence it can be used to create a maximal period
sequence.

However, for $n$ even, life is not quite so simple. No such sequence exists for $n=2$ and, for
$n=4$, $[11110010]$ is a unique example of a maximal sequence (unique up to reversal and
switching zeros and ones). Sequences for every $n\geq4$ can be derived using the inverse Lempel
homomorphism, as discussed --- using different terminology --- by Lempel~\cite[Section
III]{Lempel70}.  In fact, as per \cite[Corollary 1]{Lempel70}, there is a 1-1 correspondence
between even weight de Bruijn sequences of span $n$ and complementary pairs of sequences with
span $n+1$, and, of course, every de Bruijn sequence of span greater than 1 has even weight.
This point is discussed by Hartke~\cite{Hartke00}.

\section*{Acknowledgements}

We would like to thank the anonymous referee who provided a number of important comments that
significantly improved this paper.

%\bibliographystyle{plain}
%\bibliographystyle{amsplain}
%\bibliography{Coding}

\providecommand{\bysame}{\leavevmode\hbox to3em{\hrulefill}\thinspace}
\providecommand{\MR}{\relax\ifhmode\unskip\space\fi MR }
\providecommand{\MRhref}[2]{%
  \href{http://www.ams.org/mathscinet-getitem?mr=#1}{#2}
} \providecommand{\href}[2]{#2}

\end{document}